\documentclass{amsart}

\usepackage[T1]{fontenc}
\usepackage[utf8]{inputenc}
\usepackage[english]{babel}
\usepackage{amsmath,amssymb,amsthm,mathtools}
\usepackage{xcolor}
\usepackage{tikz}
\usepackage{float}
\usepackage{verbatim}
\usepackage{hyperref}

\numberwithin{equation}{section}

\newtheorem{theorem}{Theorem}[section]
\newtheorem{lemma}[theorem]{Lemma}
\newtheorem{definition}[theorem]{Definition}

\newtheorem{proposition}[theorem]{Proposition}

\newtheorem{question}[theorem]{Question}

\newtheorem{corollary}[theorem]{Corollary}
\theoremstyle{remark}

\newtheorem{remark}[theorem]{Remark}

\DeclareMathOperator{\supp}{supp}

\DeclareMathOperator{\topo}{top}

\DeclareMathOperator{\Per}{Per}

\begin{document}

\title[Approximation by zero-entropy systems]{Invariant measures with full support and approximation by zero-entropy systems in the $C^0$-Gromov--Hausdorff topology}

\author{Jorge Crisóstomo}
\address{Facultad de Ciencias Matemáticas, UNMSM, Lima, Perú}
\email{jcrisostomop@unmsm.edu.pe}
\thanks{}

\author{Richard Cubas}
\address{Universidad Científica del Sur, Lima, Perú}
\curraddr{}
\email{cubas.mat@usp.br}
\thanks{}

\date{}

\begin{abstract}
In this paper we prove that every homeomorphism of a compact metric space admitting an invariant probability measure with full support can be approximated in the $C^0$-Gromov--Hausdorff topology by homeomorphisms with zero topological entropy. The argument relies on the ergodic decomposition theorem and on the existence of points with dense positive orbit in the supports of suitable ergodic components. As a consequence, topological entropy is not stable under $C^0$-Gromov--Hausdorff perturbations within this class. We also show that if, in addition, the homeomorphism is topologically $GH$-stable, then its periodic points are dense in the ambient space. Finally, by combining this framework with a previous result on transitive and topologically $GH$-stable homeomorphisms, we deduce that every dynamics in this class admits an invariant measure with full support and therefore falls within the scope of the general approximation theorem by zero-entropy systems.
\end{abstract}

\maketitle

\section{Introduction}

Topological stability is one of the central notions in Dynamical Systems for describing the qualitative persistence of a dynamics under perturbations. Roughly speaking, a topologically stable system preserves its dynamical behavior, at least at the level of semi-conjugacy or conjugacy, under sufficiently small perturbations when these are considered on the same phase space. In the classical $C^0$ topology, this notion has been extensively studied and is closely related to fundamental properties such as expansiveness, shadowing, and topological conjugacy theory \cite{walters1970,walters1978,katokhasselblatt1995,robinson1999}.

A more flexible framework was introduced by Arbieto and Morales in \cite{arbieto2017topological}, who defined a $C^0$-Gromov--Hausdorff distance between dynamical systems possibly defined on different compact metric spaces. This distance combines the uniform metric between maps with the Gromov--Hausdorff distance between metric spaces \cite{gromov2007,burago2001}, and allows one to compare dynamics even when the phase spaces do not coincide. Within this context, the authors also introduced the notion of $GH$ topological stability, thereby extending the classical concept of topological stability to a more general setting \cite{arbieto2017topological}.

From this viewpoint, several questions have arisen concerning the dynamical structure induced by the $C^0$-Gromov--Hausdorff topology. One of the most natural is to determine when a given dynamics can be approximated, in the $GH^0$ sense, by simpler systems. In this direction, it was proved in \cite{cubas2018propriedades} that every topologically transitive homeomorphism of a compact metric space can be approximated in the $C^0$-Gromov--Hausdorff topology by periodic dynamics. Later, Jung \cite{jung2019closure} obtained, under additional assumptions, converse results and related characterizations in terms of chain transitivity. Complementary results and a systematic exposition of this framework can also be found in the monograph of Lee and Morales \cite{lee2022gromov}.

The purpose of this paper is to revisit this approximation problem from an ergodic viewpoint. Instead of assuming topological transitivity as the main hypothesis, we assume that the homeomorphism admits an invariant probability measure with full support. This condition is natural from the standpoint of ergodic theory and, at the same time, sufficiently robust to produce finite approximating models \cite{walters1982,viana2016foundations}. In particular, we show that this hypothesis allows one to replace the existence of a globally dense orbit by a finite family of orbit blocks associated with suitably selected ergodic components.

Our first main result states that every homeomorphism of a compact metric space admitting an invariant probability measure with full support can be approximated in the $C^0$-Gromov--Hausdorff topology by homeomorphisms with zero topological entropy.

\begin{theorem}
Let $f:X\to X$ be a homeomorphism of a compact metric space. If $f$ admits an invariant probability measure with full support, then for every $\delta>0$ there exist a compact metric space $Y_\delta$ and a homeomorphism $g_\delta:Y_\delta\to Y_\delta$ such that
\[
d_{GH^0}(f,g_\delta)<\delta
\qquad\text{and}\qquad
h_{\topo}(g_\delta)=0.
\]
\end{theorem}

The proof is based on the ergodic decomposition theorem. The full-support hypothesis makes it possible to select, at an arbitrarily fine spatial scale, a finite family of ergodic components whose supports cover the space up to a small error. On each of these supports we choose a point with dense positive orbit and, from sufficiently long segments of such orbits, we construct a finite system by closing each block into a cycle. In this way one obtains a finite dynamics which is $GH^0$-close to the original one. Since every continuous map on a finite space has zero topological entropy, the desired approximation follows.

This result shows, in particular, that topological entropy is not stable under $C^0$-Gromov--Hausdorff perturbations, even within the class of homeomorphisms admitting invariant measures with full support. In this sense, the $GH^0$ topology exhibits a markedly different behavior from that of the classical $C^0$ topology.

Our second main result concerns topologically $GH$-stable homeomorphisms. Using a periodic version of the previous construction, we prove that the combination of $GH$ topological stability with a natural ergodic hypothesis imposes a strong periodic structure on the dynamics.

\begin{theorem}
Let $f:X\to X$ be a topologically $GH$-stable homeomorphism of a compact metric space. If $f$ admits an invariant probability measure with full support, then
\[
\overline{\Per(f)}=X.
\]
\end{theorem}

This result may be interpreted as a manifestation, in the $GH^0$ context, of the general principle according to which dynamical stability forces recurrence and periodic approximation. As an immediate consequence, we obtain that minimal homeomorphisms cannot be topologically $GH$-stable under the hypotheses considered here.

Moreover, we show that the presence of dense periodic points allows one to construct invariant measures with full support. By combining this observation with a previous result from \cite{cubas2018propriedades} on the density of periodic points in the transitive and topologically $GH$-stable case, we deduce that every transitive and topologically $GH$-stable homeomorphism admits an invariant probability measure with full support. In particular, the general approximation theorem by zero-entropy systems also applies to this class.

The conceptual novelty of the present work lies in the fact that the approximation is not derived from a global hypothesis of topological transitivity, but rather from a hypothesis of ergodic nature. More precisely, the argument shows that the existence of an invariant measure with full support suffices to reconstruct, at small scales, enough recurrence distributed throughout the whole space to produce finite approximating models.

The paper is organized as follows. In Section 2 we collect the definitions and preliminary results concerning the $C^0$-Gromov--Hausdorff distance, topological entropy, invariant measures, and ergodic decomposition. In Section 3 we prove the main approximation theorem by zero-entropy systems and present a periodic version of the construction. In Section 4 we apply this scheme to the case of topologically $GH$-stable homeomorphisms, deduce the density of periodic points, and obtain an additional consequence for transitive topologically $GH$-stable homeomorphisms. Finally, in Section 5 we discuss the ergodic interpretation of the results and formulate some open questions.

\section{Preliminaries}

In this section we fix notation and collect the concepts and results that will be used in the subsequent proofs. We restrict ourselves to the tools that are strictly necessary for the development of the paper.

\subsection{The $C^0$-Gromov--Hausdorff distance}

Let $(X,d)$ be a compact metric space. If $A,B\subset X$ are nonempty subsets, the Hausdorff distance between $A$ and $B$ is defined by
\[
d_H(A,B)=\max\left\{\sup_{a\in A}d(a,B),\sup_{b\in B}d(A,b)\right\},
\]
where
\[
d(a,B)=\inf_{b\in B}d(a,b),
\qquad
d(A,b)=\inf_{a\in A}d(a,b).
\]
For the basic theory of compact metric spaces and the Hausdorff distance, see \cite{burago2001,gromov2007}.

We now recall the notion of a $\Delta$-isometry, which allows one to compare metric spaces that are not necessarily isometric.

\begin{definition}
Let $(X,d^X)$ and $(Y,d^Y)$ be compact metric spaces, and let $\Delta>0$. A map, not necessarily continuous, $i:X\to Y$ is called a \emph{$\Delta$-isometry} if
\[
\max\left\{
d_H(i(X),Y),\;
\sup_{x,x'\in X}\left|d^Y(i(x),i(x'))-d^X(x,x')\right|
\right\}<\Delta.
\]
\end{definition}

Based on this notion, the Gromov--Hausdorff distance between compact metric spaces is defined as follows \cite{gromov2007,burago2001}.

\begin{definition}
The Gromov--Hausdorff distance between two compact metric spaces $X$ and $Y$ is defined by
\[
d_{GH}(X,Y)
=
\inf\left\{
\Delta>0:\; \exists\ \Delta\text{-isometries } i:X\to Y,\ j:Y\to X
\right\}.
\]
\end{definition}

If $u,v:Z\to W$ are maps defined on the same compact metric space $(Z,d^Z)$ and taking values in the same metric space $(W,d^W)$, we define their uniform distance by
\[
d_{C^0}(u,v)=\sup_{z\in Z} d^W\bigl(u(z),v(z)\bigr).
\]
When $u$ and $v$ are continuous, this coincides with the usual $C^0$ distance.

The following definition, introduced by Arbieto and Morales in \cite{arbieto2017topological}, combines the uniform distance between maps with the Gromov--Hausdorff distance between metric spaces.

\begin{definition}
Let $f:X\to X$ and $g:Y\to Y$ be continuous maps defined on compact metric spaces $X$ and $Y$, respectively. The $C^0$-Gromov--Hausdorff distance between $f$ and $g$ is defined by
\[
d_{GH^0}(f,g)
=
\inf\Big\{
\Delta>0:\; \exists\ \Delta\text{-isometries } i:X\to Y,\ j:Y\to X
\]
\[
\hspace{2.4cm}\text{ such that }
d_{C^0}(g\circ i,i\circ f)<\Delta
\quad\text{and}\quad
d_{C^0}(j\circ g,f\circ j)<\Delta
\Big\}.
\]
\end{definition}

This notion makes it possible to compare dynamical systems defined on different metric spaces and provides the natural framework for studying approximation phenomena in the dynamical Gromov--Hausdorff sense \cite{arbieto2017topological,lee2022gromov}.

\begin{definition}
Let $\mathcal C$ be a family of continuous maps defined on compact metric spaces. We say that a continuous map $f:X\to X$ is \emph{$C^0$-Gromov--Hausdorff approximated} by elements of $\mathcal C$ if there exists a sequence $(g_n)_{n\in\mathbb N}\subset \mathcal C$ such that
\[
\lim_{n\to\infty}d_{GH^0}(f,g_n)=0.
\]
\end{definition}

In particular, as we shall see later, the topology induced by $d_{GH^0}$ allows one to approximate a dynamics defined on a compact space by dynamics defined on finite spaces.

\subsection{Topological entropy}

Let $f:X\to X$ be a continuous map defined on a compact metric space $(X,d)$. We recall the definition of topological entropy in terms of separated sets; see, for instance, \cite{walters1982,katokhasselblatt1995,robinson1999}. Given $n\in\mathbb N$ and $\delta>0$, a finite subset $E\subset X$ is said to be \emph{$(n,\delta)$-separated} if for every $x,y\in E$, with $x\neq y$, one has
\[
\max_{0\leq j\leq n-1} d\big(f^j(x),f^j(y)\big)\geq \delta.
\]

We denote by
\[
s_n(f,\delta)
=
\sup\left\{\#E:\; E\subset X \text{ is an } (n,\delta)\text{-separated set}\right\}.
\]

\begin{definition}
The topological entropy of $f$ is defined by
\[
h_{\topo}(f)
=
\lim_{\delta\to 0}
\left(
\limsup_{n\to\infty}\frac{1}{n}\log s_n(f,\delta)
\right).
\]
\end{definition}

The only entropy property that we shall use explicitly is the following elementary fact.

\begin{lemma}\label{lem:finite_entropy_zero}
If $Y$ is a finite set and $g:Y\to Y$ is a map, then
\[
h_{\topo}(g)=0.
\]
\end{lemma}

\begin{proof}
Since $Y$ is finite, for every $n\in\mathbb N$ and every $\delta>0$ one has
\[
s_n(g,\delta)\leq \#Y.
\]
Therefore,
\[
\frac{1}{n}\log s_n(g,\delta)\leq \frac{1}{n}\log(\#Y).
\]
Taking the limit superior as $n\to\infty$, we obtain
\[
\limsup_{n\to\infty}\frac{1}{n}\log s_n(g,\delta)=0.
\]
Finally, letting $\delta\to 0$, we conclude that
\[
h_{\topo}(g)=0.
\]
\end{proof}

\subsection{Invariant measures and ergodic decomposition}

Let $f:X\to X$ be a homeomorphism defined on a compact metric space $X$. We say that a Borel probability measure $\mu$ on $X$ is \emph{invariant} under $f$ if
\[
\mu(f^{-1}(E))=\mu(E)
\]
for every measurable set $E\subset X$.

\begin{definition}
The support of a Borel probability measure $\mu$ on $X$ is defined by
\[
\supp(\mu)=\left\{x\in X:\; \mu(U)>0 \text{ for every open neighborhood } U \text{ of } x\right\}.
\]
We say that $\mu$ has \emph{full support} if
\[
\supp(\mu)=X.
\]
\end{definition}

\begin{definition}
An invariant measure $\mu$ is said to be \emph{ergodic} if for every measurable set $A\subset X$ such that $f^{-1}(A)=A$, one has
\[
\mu(A)\in\{0,1\}.
\]
\end{definition}

The basic theory of invariant measures, the Krylov--Bogolyubov theorem, Birkhoff's ergodic theorem, and ergodic decomposition may be found in \cite{walters1982,viana2016foundations}. We shall use the following ergodic decomposition theorem.

\begin{theorem}[Ergodic decomposition]\label{thm:ergodic_decomposition}
Let $X$ be a compact metric space and let $f:X\to X$ be a homeomorphism. If $\mu$ is an $f$-invariant probability measure, then there exist a measurable set $X_0\subset X$ with $\mu(X_0)=1$, a measurable partition $\mathcal P$ of $X_0$, and a family of probability measures $\{\mu_P\}_{P\in\mathcal P}$ such that:
\begin{itemize}
    \item[(i)] for $\hat{\mu}$-almost every class $P\in\mathcal P$, one has $\mu_P(P)=1$;
    \item[(ii)] for every measurable set $E\subset X$, the map
    \[
    P\mapsto \mu_P(E)
    \]
    is measurable;
    \item[(iii)] for $\hat{\mu}$-almost every class $P\in\mathcal P$, the measure $\mu_P$ is invariant and ergodic;
    \item[(iv)] for every measurable set $E\subset X$,
    \[
    \mu(E)=\int \mu_P(E)\,d\hat{\mu}(P).
    \]
\end{itemize}
\end{theorem}

The following standard fact will be fundamental in the construction of the approximating systems. We include it for completeness, since it is precisely the point at which ergodicity is translated into topological transitivity on the support.

\begin{lemma}\label{lem:dense_orbit_support}
Let $f:X\to X$ be a continuous map on a compact metric space, and let $\mu$ be an invariant ergodic probability measure. Then there exists a point $\omega \in \supp(\mu)$ such that
\[
\overline{\mathcal O_f^+(\omega)}=\supp(\mu),
\]
where
\[
\mathcal O_f^+(\omega)=\{\omega,f(\omega),f^2(\omega),\dots\}.
\]
In particular, the restriction $f|_{\supp(\mu)}$ is topologically transitive.
\end{lemma}

\begin{proof}
Let $\{U_n\}_{n\in\mathbb N}$ be a countable basis of relatively open subsets of $\supp(\mu)$. Since each $U_n$ is a nonempty relatively open subset of $\supp(\mu)$, one has
\[
\mu(U_n)>0
\qquad\text{for every }n\in\mathbb N.
\]

By Birkhoff's ergodic theorem applied to the indicator function $\mathbf 1_{U_n}$, for each $n\in\mathbb N$ there exists a set $A_n\subset X$ with $\mu(A_n)=1$ such that, for every $x\in A_n$,
\[
\lim_{N\to\infty}\frac1N\sum_{j=0}^{N-1}\mathbf 1_{U_n}(f^j(x))
=
\int \mathbf 1_{U_n}\,d\mu
=
\mu(U_n)>0.
\]
Consequently, for every $x\in A_n$ there exists some integer $k\ge 0$ such that
$f^k(x)\in U_n$. In fact, the positive orbit of $x$ enters $U_n$ infinitely many times.

Define $A=\bigcap_{n\in\mathbb N}A_n$. Then $\mu(A)=1$. Since moreover $\mu(\supp(\mu))=1$, there exists
\[
\omega\in A\cap \supp(\mu).
\]

Now let $U\subset \supp(\mu)$ be any nonempty relatively open set. Then there exists $n\in\mathbb N$ such that
\[
U_n\subset U
\qquad\text{and}\qquad
U_n\neq\emptyset.
\]
Since $\omega\in A_n$, there exists $k\ge 0$ such that $f^k(\omega)\in U_n\subset U$.
Therefore, the positive orbit of $\omega$ intersects every nonempty relatively open subset of $\supp(\mu)$. This implies that
\[
\overline{\mathcal O_f^+(\omega)}=\supp(\mu).
\]

The last assertion is immediate, since the existence of a dense positive orbit in $\supp(\mu)$ means precisely that $f|_{\supp(\mu)}$ is topologically transitive.
\end{proof}

\section{Approximation by zero-entropy systems}

In this section we prove the main result of the paper. The goal is to show that the existence of an invariant probability measure with full support makes it possible to construct, at every prescribed scale, a finite dynamical system that approximates the original homeomorphism in the $C^0$-Gromov--Hausdorff topology. Since every dynamics defined on a finite space has zero topological entropy, the desired approximation follows.

\begin{theorem}\label{thm:main}
Let $f:X\to X$ be a homeomorphism of a compact metric space. If $f$ admits an invariant probability measure with full support, then for every $\delta>0$ there exist a compact metric space $Y_\delta$ and a homeomorphism $g_\delta:Y_\delta\to Y_\delta$ such that
\[
d_{GH^0}(f,g_\delta)<\delta
\qquad\text{and}\qquad
h_{\topo}(g_\delta)=0.
\]
\end{theorem}

\begin{proof}
Let $\mu$ be an $f$-invariant probability measure with full support. Fix $\delta>0$. Since $f$ is uniformly continuous and $X$ is compact, there exists $\beta>0$ such that
\begin{eqnarray}\label{ecu0}
d(x,y)<\beta
\quad\Longrightarrow\quad
d\bigl(f(x),f(y)\bigr)<\frac{\delta}{3}
\qquad\text{for all }x,y\in X.
\end{eqnarray}
Take $0<\alpha<\min \left\{\frac{\delta}{9},\frac{\beta}{2}\right\}$. By compactness of $X$, there exists a finite set $\{x_1,\dots,x_r\}\subset X$ such that
\[
d_H(\{x_1,\dots,x_r\},X)<\alpha.
\]
Let $\{\mu_P\}_{P\in\mathcal P}$ be the ergodic decomposition of $\mu$ as in Theorem \ref{thm:ergodic_decomposition}. Since $\supp(\mu)=X$, one has
\[
\mu\bigl(B(x_i,\alpha)\bigr)>0
\qquad\text{for every } i=1,\dots,r.
\]
Applying Theorem \ref{thm:ergodic_decomposition}, for each $i\in\{1,\dots,r\}$ we obtain
\[
\mu\bigl(B(x_i,\alpha)\bigr)
=
\int \mu_P\bigl(B(x_i,\alpha)\bigr)\,d\hat\mu(P)>0.
\]
Since $\mu_P$ is ergodic for $\hat\mu$-almost every class $P\in\mathcal P$, there exists a class $P_i\in\mathcal P$ such that
\begin{eqnarray}\label{ecu1}
\mu_{P_i}\bigl(B(x_i,\alpha)\bigr)>0,
\end{eqnarray}
and $\mu_{P_i}$ is ergodic. Denote this measure by $\mu_i$. From \eqref{ecu1}, for every $i=1,\ldots,r$, one has
\begin{eqnarray}\label{ecu2}
B(x_i,\alpha)\cap \supp(\mu_i)\neq\emptyset.
\end{eqnarray}

By Lemma \ref{lem:dense_orbit_support}, for each $i$ there exists a point $z_i\in \supp(\mu_i)$ such that
\[
\overline{\mathcal O_f^+(z_i)}=\supp(\mu_i).
\]
Since $B(x_i,\alpha)\cap \supp(\mu_i)\neq\emptyset$ and the positive orbit of $z_i$ is dense in $\supp(\mu_i)$, there exists an integer $\ell_i\ge0$ such that
\[
f^{\ell_i}(z_i)\in B(x_i,\alpha).
\]
Define
\[
a_i:=f^{\ell_i}(z_i).
\]
Since $f^{\ell_i}$ is a homeomorphism and $\supp(\mu_i)$ is $f$-invariant, one has
\begin{equation}\label{ecu3}
\overline{\mathcal O_f^+(a_i)}
=
\overline{f^{\ell_i}\bigl(\mathcal O_f^+(z_i)\bigr)}
=
f^{\ell_i}\bigl(\overline{\mathcal O_f^+(z_i)}\bigr)
=
f^{\ell_i}\bigl(\supp(\mu_i)\bigr)
=
\supp(\mu_i).
\end{equation}
In particular,
\[
a_i\in \supp(\mu_i)\cap B(x_i,\alpha).
\]
Since $\supp(\mu_i)$ is $f$-invariant and $a_i\in \supp(\mu_i)$, one has
\[
f^{-1}(a_i)\in \supp(\mu_i).
\]
By \eqref{ecu3}, since $f^{-1}(a_i)\in \supp(\mu_i)$, there exists an integer $N_i\ge0$ such that
\[
d\bigl(f^{N_i}(a_i),f^{-1}(a_i)\bigr)<\beta.
\]
Applying $f$ and using the choice of $\beta$ in \eqref{ecu0}, we obtain
\[
d\bigl(f^{N_i+1}(a_i),a_i\bigr)
=
d\bigl(f(f^{N_i}(a_i)),f(f^{-1}(a_i))\bigr)
<
\frac{\delta}{3}.
\]

We now define the finite set
\[
Y_\delta=\{(i,k):\,1\le i\le r,\ 0\le k\le N_i\}.
\]
On $Y_\delta$ we consider the map
\[
q:Y_\delta\to X,
\qquad
q(i,k)=f^k(a_i).
\]

To avoid possible identifications between distinct pairs $(i,k)$ and $(j,\ell)$ having the same image in $X$, we equip $Y_\delta$ with the metric
\[
\rho(u,v)=d\bigl(q(u),q(v)\bigr)+\alpha\,\eta(u,v),
\qquad u,v\in Y_\delta,
\]
where
\[
\eta(u,v)=
\begin{cases}
0,& u=v,\\
1,& u\neq v.
\end{cases}
\]
Since $d(q(u),q(v))$ defines a pseudometric on $Y_\delta$ and $\alpha\,\eta(u,v)$ is a metric, the function $\rho$ is a metric on $Y_\delta$. In particular, $(Y_\delta,\rho)$ is a compact metric space.

We now define $g_\delta:Y_\delta\to Y_\delta$ by
\[
g_\delta(i,k)=(i,k+1), \qquad 0\le k\le N_i-1,
\]
and
\[
g_\delta(i,N_i)=(i,0).
\]
Clearly, $g_\delta$ is a permutation of $Y_\delta$; since $Y_\delta$ is finite, it follows that $g_\delta$ is a homeomorphism. Moreover, for every $u=(i,k)\in Y_\delta$ one has
\begin{eqnarray}\label{ecu1.2}
d\bigl(q(g_\delta(u)),f(q(u))\bigr)<\frac{\delta}{3}.
\end{eqnarray}
Indeed, if $0\le k\le N_i-1$, then
\[
q(g_\delta(i,k))=q(i,k+1)=f^{k+1}(a_i)=f(q(i,k)),
\]
whereas if $k=N_i$, then
\[
d\bigl(q(g_\delta(i,N_i)),f(q(i,N_i))\bigr)
=
d\bigl(a_i,f^{N_i+1}(a_i)\bigr)
<
\frac{\delta}{3}.
\]
On the other hand, since $\{a_1,\dots,a_r\}\subset q(Y_\delta)$, we have
\[
d_H(q(Y_\delta),X)\le d_H(\{a_1,\dots,a_r\},X).
\]
Moreover,
\[
d_H(\{a_1,\dots,a_r\},X)
\le
 d_H(\{a_1,\dots,a_r\},\{x_1,\dots,x_r\})
+
d_H(\{x_1,\dots,x_r\},X)
<
2\alpha,
\]
since $a_i\in B(x_i,\alpha)$ for every $i$. Therefore,
\begin{equation}\label{ecu6}
d_H(q(Y_\delta),X)<2\alpha<\delta.
\end{equation}

Furthermore, if $u\neq v$, then
\[
\rho(u,v)=d\bigl(q(u),q(v)\bigr)+\alpha.
\]
Therefore, for every $u,v\in Y_\delta$ one has $\bigl|d\bigl(q(u),q(v)\bigr)-\rho(u,v)\bigr|\le \alpha$. 
Hence,
\begin{eqnarray}\label{ecu7}
\sup_{u,v\in Y_\delta}
\bigl|d\bigl(q(u),q(v)\bigr)-\rho(u,v)\bigr|
\le \alpha<\delta.
\end{eqnarray}
From \eqref{ecu6} and \eqref{ecu7} we conclude that
\[
q:(Y_\delta,\rho)\to X
\]
is a $\delta$-isometry. Moreover,
\[
d_{C^0}(q\circ g_\delta,f\circ q)<\frac{\delta}{3}<\delta.
\]

Since $d_H(q(Y_\delta),X)<2\alpha$, we construct a map $j:X\to Y_\delta$ as follows: for each $x\in X$, choose a point $j(x)\in Y_\delta$ such that
\begin{eqnarray}\label{ecu6.1}
d\bigl(x,q(j(x))\bigr)<2\alpha.
\end{eqnarray}
We claim that $j$ is a $\delta$-isometry. Given $x,x'\in X$, one has
\[
\bigl|d\bigl(q(j(x)),q(j(x'))\bigr)-d(x,x')\bigr|
\le
 d\bigl(x,q(j(x))\bigr)+d\bigl(x',q(j(x'))\bigr)
<
4\alpha.
\]
Moreover, by the definition of $\rho$,
\[
\bigl|\rho(j(x),j(x'))-d\bigl(q(j(x)),q(j(x'))\bigr)\bigr|
\le \alpha.
\]
Therefore,
\begin{eqnarray}\label{ecu8}
\bigl|\rho(j(x),j(x'))-d(x,x')\bigr|
<
5\alpha
<
\delta,
\end{eqnarray}
since $\alpha<\delta/9$. On the other hand, for every $y\in Y_\delta$,
\[
\rho\bigl(j(q(y)),y\bigr)
\le
 d\bigl(q(j(q(y))),q(y)\bigr)+\alpha
<
3\alpha.
\]
Therefore,
\begin{eqnarray}\label{ecu9}
d_H\bigl(j(X),Y_\delta\bigr)<3\alpha<\delta.
\end{eqnarray}
From \eqref{ecu8} and \eqref{ecu9}, we conclude that $j:X\to (Y_\delta,\rho)$
is a $\delta$-isometry.
\vspace{0.3cm}

We now verify that $d_{C^0}(j\circ f,g_\delta\circ j)<\delta$. Indeed, let $x\in X$. By the definition of $\rho$, one has
\[
\rho\bigl(j(f(x)),g_\delta(j(x))\bigr)
\le
 d\bigl(q(j(f(x))),q(g_\delta(j(x)))\bigr)+\alpha.
\]
Moreover,
\[
\begin{aligned}
d\bigl(q(j(f(x))),q(g_\delta(j(x)))\bigr)
&\le
 d\bigl(q(j(f(x))),f(x)\bigr)
+d\bigl(f(x),f(q(j(x)))\bigr)\\
&\qquad
+d\bigl(f(q(j(x))),q(g_\delta(j(x)))\bigr).
\end{aligned}
\]
By the definition of $j$ given in \eqref{ecu6.1}, the first term on the right-hand side is smaller than $2\alpha$. For the second term, since
\[
d\bigl(x,q(j(x))\bigr)<2\alpha<\beta,
\]
the choice of $\beta$ in \eqref{ecu0} implies that $d\bigl(f(x),f(q(j(x)))\bigr)<\frac{\delta}{3}$.
Finally, by \eqref{ecu1.2}, the third term is smaller than $\frac{\delta}{3}$. Consequently,
\[
\rho\bigl(j(f(x)),g_\delta(j(x))\bigr)
<
2\alpha+\frac{\delta}{3}+\frac{\delta}{3}+\alpha
=
3\alpha+\frac{2\delta}{3}
<
3\left(\frac{\delta}{9}\right)+\frac{2\delta}{3}
=
\delta.
\]
Therefore,
\[
d_{C^0}(j\circ f,g_\delta\circ j)<\delta.
\]

We have constructed $\delta$-isometries
\[
q:(Y_\delta,\rho)\to X
\qquad\text{and}\qquad
j:X\to (Y_\delta,\rho)
\]
such that
\[
d_{C^0}(q\circ g_\delta,f\circ q)<\delta
\qquad\text{and}\qquad
d_{C^0}(j\circ f,g_\delta\circ j)<\delta.
\]
By the definition of the distance $d_{GH^0}$, it follows that
\[
d_{GH^0}(f,g_\delta)<\delta.
\]

Finally, since $Y_\delta$ is finite, Lemma \ref{lem:finite_entropy_zero} implies that
\[
h_{\topo}(g_\delta)=0.
\]
This completes the proof.
\end{proof}

The previous construction also has an additional property that will be useful in the next section.

\begin{proposition}\label{prop:periodic_version}
Under the assumptions of Theorem \ref{thm:main}, the approximating homeomorphism $g_\delta$ may be chosen in such a way that every point of $Y_\delta$ is periodic.
\end{proposition}

\begin{proof}
The construction carried out in the proof of Theorem \ref{thm:main} already has this property. Indeed, for each $i=1,\dots,r$, the subset
\[
\{(i,0),(i,1),\dots,(i,N_i)\}
\]
is invariant under $g_\delta$, and the restriction of $g_\delta$ to this set is a finite periodic orbit. Therefore, every point of $Y_\delta$ is periodic for $g_\delta$.
\end{proof}

\begin{remark}\label{rem:single_cycle}
The previous construction yields an approximating homeomorphism whose phase space is a finite union of periodic orbits. In particular, every point of $Y_\delta$ is periodic. This property will be sufficient for the subsequent applications.
\end{remark}

\section{Consequences for topologically $GH$-stable systems}

In this section we analyze the dynamical consequences of the approximation constructed in the previous section when the additional hypothesis of $GH$ topological stability is imposed. The central idea is that if a periodic system is sufficiently close, in the $C^0$-Gromov--Hausdorff topology, to a topologically $GH$-stable homeomorphism, then the semi-conjugacy provided by stability sends periodic points of the approximating system to periodic points of the original one. By combining this principle with the approximation by finite systems obtained in Theorem \ref{thm:main}, one deduces the density of periodic points.

\subsection{$GH$ topological stability}

We recall the notion of $GH$ topological stability introduced by Arbieto and Morales in \cite{arbieto2017topological}. This definition extends the classical notion of topological stability to the context in which the perturbed dynamics may be defined on different compact metric spaces.

\begin{definition}
Let $f:X\to X$ be a homeomorphism of a compact metric space. We say that $f$ is \emph{topologically $GH$-stable} if for every $\varepsilon>0$ there exists $\eta>0$ such that, for every homeomorphism $g:Y\to Y$ of a compact metric space satisfying
\[
d_{GH^0}(f,g)<\eta,
\]
there exists a continuous $\varepsilon$-isometry $h:Y\to X$ such that
\[
f\circ h=h\circ g.
\]
\end{definition}

The map $h$ appearing in the previous definition provides a semi-conjugacy between $g$ and $f$, and will be the fundamental tool for transferring periodicity from the approximating systems to the original system.

\subsection{Density of periodic points}

We are now in a position to obtain the main application of the present work. Theorem \ref{thm:main}, together with Proposition \ref{prop:periodic_version}, produces periodic systems arbitrarily close to $f$ in the $GH^0$ topology. The $GH$ topological stability then allows one to transfer periodicity from these approximating systems to the original dynamics.

\begin{theorem}\label{thm:denseperiodic}
Let $f:X\to X$ be a topologically $GH$-stable homeomorphism of a compact metric space. If $f$ admits an invariant probability measure with full support, then
\[
\overline{\Per(f)}=X.
\]
\end{theorem}

\begin{proof}
Let $\varepsilon>0$. Since $f$ is topologically $GH$-stable, there exists $\eta>0$ such that for every homeomorphism $g:Y\to Y$ with
\[
d_{GH^0}(f,g)<\eta,
\]
there exists a continuous $\varepsilon$-isometry
\[
h:Y\to X
\]
satisfying
\[
f\circ h=h\circ g.
\]

We now apply Proposition \ref{prop:periodic_version} at scale $\eta$. We obtain a compact metric space $Y_\eta$ and a homeomorphism
\[
g_\eta:Y_\eta\to Y_\eta
\]
such that
\[
d_{GH^0}(f,g_\eta)<\eta
\]
and, moreover, every point of $Y_\eta$ is periodic for $g_\eta$.

By the $GH$ topological stability of $f$, there then exists a continuous $\varepsilon$-isometry
\[
h:Y_\eta\to X
\]
such that
\[
f\circ h=h\circ g_\eta.
\]

Let $y\in Y_\eta$. Since $y$ is periodic for $g_\eta$, there exists an integer $m\ge 1$ such that
\[
g_\eta^m(y)=y.
\]
Applying $h$ and using the semi-conjugacy relation, we obtain
\[
f^m(h(y))
=
h(g_\eta^m(y))
=
h(y).
\]
Therefore,
\[
h(y)\in \Per(f).
\]
Since $y\in Y_\eta$ was arbitrary, we conclude that
\[
h(Y_\eta)\subset \Per(f).
\]

On the other hand, since $h$ is an $\varepsilon$-isometry, in particular one has
\[
d_H(h(Y_\eta),X)<\varepsilon.
\]
It follows that for every $x\in X$ there exists $p\in h(Y_\eta)\subset \Per(f)$ such that
\[
d(x,p)<\varepsilon.
\]
Since $\varepsilon>0$ is arbitrary, we conclude that
\[
\overline{\Per(f)}=X.
\]
\end{proof}

The previous theorem shows that the combination of $GH$ topological stability with the existence of an invariant measure having full support imposes a strong periodic structure on the dynamics. In particular, on infinite compact spaces, this condition is incompatible with minimality.

\subsection{A corollary for minimal homeomorphisms}

As an immediate consequence of Theorem \ref{thm:denseperiodic}, we obtain the following obstruction to $GH$ topological stability in the minimal setting.

\begin{corollary}\label{cor:minimal_not_GHstable}
Let $f:X\to X$ be a minimal homeomorphism of an infinite compact metric space. Then $f$ cannot be topologically $GH$-stable.
\end{corollary}

\begin{proof}
Every continuous homeomorphism of a compact metric space admits an invariant probability measure \cite{walters1982}. Let $\mu$ be one such measure. Since $\supp(\mu)$ is a nonempty compact $f$-invariant set, the minimality of $f$ necessarily implies that
\[
\supp(\mu)=X.
\]

Suppose now, by contradiction, that $f$ is topologically $GH$-stable. Then, by Theorem \ref{thm:denseperiodic},
\[
\overline{\Per(f)}=X.
\]
But a minimal homeomorphism on an infinite compact space cannot have periodic points. Indeed, if $p$ were periodic, then its orbit would be finite, closed, and invariant; by minimality, that orbit would have to coincide with the whole space $X$, contradicting the fact that $X$ is infinite. This contradiction proves that $f$ cannot be topologically $GH$-stable.
\end{proof}

\subsection{Invariant measures with full support from dense periodic points}

The following lemma will be useful for relating our results to the transitive case previously treated in \cite{cubas2018propriedades}. The statement is completely general and shows that the density of periodic points is sufficient to produce an invariant measure with full support.

\begin{lemma}\label{lem:denseperiodic_fullsupport}
Let $f:X\to X$ be a homeomorphism of a compact metric space. If
\[
\overline{\Per(f)}=X,
\]
then $f$ admits an invariant probability measure with full support.
\end{lemma}

\begin{proof}
Since $X$ is compact metric, it is in particular separable. Because
\[
\overline{\Per(f)}=X,
\]
there exists a dense sequence
\[
\{p_n\}_{n\in\mathbb N}\subset \Per(f).
\]
For each $n\in\mathbb N$, let $\pi_n\ge 1$ denote the period of $p_n$, and define the periodic measure associated with its orbit by
\[
\mu_n
=
\frac{1}{\pi_n}\sum_{k=0}^{\pi_n-1}\delta_{f^k(p_n)}.
\]
Each $\mu_n$ is a Borel probability measure invariant under $f$.

Now consider the measure
\[
\mu=\sum_{n=1}^{\infty}2^{-n}\mu_n.
\]
Since
\[
\sum_{n=1}^{\infty}2^{-n}=1,
\]
the measure $\mu$ is a Borel probability measure. Moreover, by linearity of the pushforward operation and the invariance of each $\mu_n$, we obtain
\[
f_*\mu
=
\sum_{n=1}^{\infty}2^{-n}f_*\mu_n
=
\sum_{n=1}^{\infty}2^{-n}\mu_n
=
\mu.
\]
Therefore, $\mu$ is invariant under $f$.

Finally, let us show that $\supp(\mu)=X$. Let $U\subset X$ be a nonempty open set. Since the sequence $\{p_n\}_{n\in\mathbb N}$ is dense, there exists $n_0\in\mathbb N$ such that
\[
p_{n_0}\in U.
\]
Since $p_{n_0}$ belongs to the periodic orbit supporting $\mu_{n_0}$, it follows that
\[
\mu_{n_0}(U)>0.
\]
Consequently,
\[
\mu(U)\ge 2^{-n_0}\mu_{n_0}(U)>0.
\]
We have proved that every nonempty open set has positive measure with respect to $\mu$, which is equivalent to
\[
\supp(\mu)=X.
\]
\end{proof}

The usefulness of the previous lemma is that it converts density-of-periodic-points results into statements of ergodic character.

\begin{proposition}\label{prop:transitive_full_support_GH}
Let $f:X\to X$ be a topologically transitive and topologically $GH$-stable homeomorphism of a compact metric space. Then $f$ admits an invariant probability measure with full support.
\end{proposition}

\begin{proof}
By \cite[Theorem 3.1.4]{cubas2018propriedades}, one has
\[
\overline{\Per(f)}=X.
\]
Applying Lemma \ref{lem:denseperiodic_fullsupport}, we conclude that $f$ admits an invariant probability measure with full support.
\end{proof}

The previous proposition shows that, within the class of transitive and topologically $GH$-stable homeomorphisms, the ergodic hypothesis of Theorem \ref{thm:main} is automatic.

\begin{corollary}\label{cor:transitive_GH_zero_entropy}
Let $f:X\to X$ be a topologically transitive and topologically $GH$-stable homeomorphism of a compact metric space. Then for every $\delta>0$ there exist a compact metric space $Y_\delta$ and a homeomorphism $g_\delta:Y_\delta\to Y_\delta$ such that
\[
d_{GH^0}(f,g_\delta)<\delta
\qquad\text{and}\qquad
h_{\topo}(g_\delta)=0.
\]
\end{corollary}

\begin{proof}
By Proposition \ref{prop:transitive_full_support_GH}, the homeomorphism $f$ admits an invariant probability measure with full support. The result then follows immediately from Theorem \ref{thm:main}.
\end{proof}

\section{Ergodic interpretation and final remarks}

Theorem \ref{thm:main} shows that the existence of an invariant measure with full support is sufficient to construct finite approximations of the dynamics in the $C^0$-Gromov--Hausdorff topology. In this sense, the result may be interpreted as a principle of global discretization: although the system need not be transitive, the ergodic information distributed throughout the whole space makes it possible to produce finite approximating models.

From the entropy viewpoint, it also follows that topological entropy is not stable under $C^0$-Gromov--Hausdorff perturbations within the class under consideration, since every dynamics admitting an invariant measure with full support can be approximated by systems with zero topological entropy.

On the other hand, Theorem \ref{thm:denseperiodic} shows that, under $GH$ topological stability, this periodic approximation translates into density of periodic points for the original dynamics. This highlights a close relationship between stability, finite approximation, and periodic recurrence.

Finally, Proposition \ref{prop:transitive_full_support_GH} shows that, in the class of transitive and topologically $GH$-stable homeomorphisms, the hypothesis of the existence of an invariant measure with full support arises naturally. In this way, the ergodic approach developed here recovers and extends previous results obtained in the transitive setting.

We conclude with a few questions that arise naturally from this work.

\begin{question}
Can Theorem \ref{thm:main} be extended to the case of homeomorphisms admitting invariant measures whose support is not full?
\end{question}

\begin{question}
What can be said about the behavior of measures of maximal entropy under $C^0$-Gromov--Hausdorff perturbations?
\end{question}

\begin{question}
Is it possible to formulate a more directly ergodic version of a closing lemma in the $GH^0$ topology?
\end{question}

\bibliographystyle{abbrv}
\bibliography{references}

@article{arbieto2017topological,
  author  = {Arbieto, Alexander and Morales Rojas, C. A.},
  title   = {Topological Stability from {Gromov--Hausdorff} Viewpoint},
  journal = {Discrete and Continuous Dynamical Systems},
  volume  = {37},
  number  = {7},
  pages   = {3531--3544},
  year    = {2017},
  doi     = {10.3934/dcds.2017151}
}

@book{burago2001,
  author    = {Burago, Dmitri and Burago, Yuri and Ivanov, Sergei},
  title     = {A Course in Metric Geometry},
  series    = {Graduate Studies in Mathematics},
  volume    = {33},
  publisher = {American Mathematical Society},
  address   = {Providence, RI},
  year      = {2001},
  isbn      = {978-0-8218-2129-9}
}

@mastersthesis{cubas2018propriedades,
  author   = {Cubas Becerra, Richard Javier},
  title    = {Propriedades de um Homeomorfismo {GH} Est{\'a}vel},
  school   = {Universidade Federal de Uberl{\^a}ndia},
  address  = {Uberl{\^a}ndia},
  year     = {2018}
}

@book{gromov2007,
  author    = {Gromov, Mikhail},
  title     = {Metric Structures for Riemannian and Non-Riemannian Spaces},
  editor    = {Lafontaine, Jacques and Pansu, Pierre},
  series    = {Modern Birkh\"auser Classics},
  publisher = {Birkh\"auser},
  address   = {Boston},
  year      = {2007}
}

@article{jung2019closure,
  author  = {Jung, Woochul},
  title   = {The closure of periodic orbits in the {Gromov--Hausdorff} space},
  journal = {Topology and its Applications},
  volume  = {264},
  pages   = {493--497},
  year    = {2019},
  doi     = {10.1016/j.topol.2019.06.048}
}

@book{katokhasselblatt1995,
  author    = {Katok, Anatole and Hasselblatt, Boris},
  title     = {Introduction to the Modern Theory of Dynamical Systems},
  series    = {Encyclopedia of Mathematics and its Applications},
  volume    = {54},
  publisher = {Cambridge University Press},
  address   = {Cambridge},
  year      = {1995}
}

@book{lee2022gromov,
  author    = {Lee, Jihoon and Morales, Carlos},
  title     = {Gromov--Hausdorff Stability of Dynamical Systems and Applications to PDEs},
  series    = {SpringerBriefs in Mathematics},
  publisher = {Springer},
  address   = {Cham},
  year      = {2022}
}

@book{robinson1999,
  author    = {Robinson, Clark},
  title     = {Dynamical Systems: Stability, Symbolic Dynamics, and Chaos},
  edition   = {2},
  publisher = {CRC Press},
  address   = {Boca Raton, FL},
  year      = {1999}
}

@book{viana2016foundations,
  author    = {Viana, Marcelo and Oliveira, Krerley},
  title     = {Foundations of Ergodic Theory},
  series    = {Cambridge Studies in Advanced Mathematics},
  volume    = {151},
  publisher = {Cambridge University Press},
  address   = {Cambridge},
  year      = {2016}
}

@article{walters1970,
  author  = {Walters, Peter},
  title   = {Anosov diffeomorphisms are topologically stable},
  journal = {Topology},
  volume  = {9},
  number  = {1},
  pages   = {71--78},
  year    = {1970},
  doi     = {10.1016/0040-9383(70)90051-0}
}

@incollection{walters1978,
  author    = {Walters, Peter},
  title     = {On the pseudo-orbit tracing property and its relationship to stability},
  booktitle = {The Structure of Attractors in Dynamical Systems},
  series    = {Lecture Notes in Mathematics},
  volume    = {668},
  pages     = {231--244},
  publisher = {Springer},
  address   = {Berlin},
  year      = {1978}
}

@book{walters1982,
  author    = {Walters, Peter},
  title     = {An Introduction to Ergodic Theory},
  series    = {Graduate Texts in Mathematics},
  volume    = {79},
  publisher = {Springer-Verlag},
  address   = {New York},
  year      = {1982},
  isbn      = {978-0-387-95152-2}
}

\end{document}